\documentclass[12pt]{amsart}
\usepackage{amssymb}
\usepackage{amsmath}
\usepackage{versions}
\usepackage[usenames,dvipsnames] {color}

\def\gam{\Gamma}
\def\d{\mathbb{D}}

\def\t{\mathbb{T}}
\def\r{\mathbb{R}}
\def\Z{\mathbb{Z}}
\def\ctwo{\mathbb{C}^2}

\def\royal{\mathcal{R}}
\def\rnk{\mathcal{R}^{n,k}}

\def\be{\begin{equation}}
\def\ee{\end{equation}}

\def\set#1#2{\{ #1 \, : \, #2\}}

\def\gaminn{\Gamma\text{-inner}}
\def\s0{s_0}
\def\p0{p_0}
\let\phi\varphi
\let\epsilon\varepsilon

\newtheorem{theorem}{Theorem}[section]

\newtheorem{lemma}[theorem]{Lemma}
\newtheorem{proposition}[theorem]{Proposition}

\numberwithin{equation}{section}
\newtheorem{defin}[theorem]{Definition}

\newtheorem{prop}[theorem]{Proposition}

\newtheorem{example}[theorem]{Example}

\newtheorem{definition}[theorem]{Definition}

\newtheorem{remark}[theorem]{Remark}

\newtheorem{fact*}[theorem]{Fact}
\DeclareMathOperator\hol{Hol}

\DeclareMathOperator\re{Re}

\newcommand\e{\mathrm{e}}

\newcommand\ii{\mathrm{i}}

\newcommand{\M}{b\Gamma}

\newcommand{\T}{\mathbb{T}}

\newcommand{\D}{\mathbb{D}}
\newcommand{\C}{\mathbb{C}}
\newcommand{\G}{\mathcal{G}}
\newcommand\Ga{\Gamma}

\newcommand{\R}{\mathbb{R}}

\newcommand{\twopartdef}[4]
{
	\left\{
		\begin{array}{ll}
			#1 &  #2 \\  \\
			#3 &  #4
		\end{array}
	\right.
}

\newcommand\al{\alpha}

\newcommand\la{\lambda}

\newcommand\si{\sigma}

\newcommand\half{{\tfrac{1}{2}}}

\newcommand\df{\stackrel{\rm def}{=}}

\newcommand\beq{\begin{equation}}

\newcommand\eeq{\end{equation}}

\newcommand\bbm{\begin{bmatrix}}
\newcommand\ebm{\end{bmatrix}}
\newcommand\bpm{\begin{pmatrix}}
\newcommand\epm{\end{pmatrix}}
\newcommand\ord{\mathrm{ord}}

\let\phi\varphi
\numberwithin{equation}{section}
\begin{document}
\title[Rational $\Gamma$-inner functions]{Algebraic and geometric aspects of rational $\Gamma$-inner functions}
\author{Jim Agler, Zinaida A. Lykova and N. J. Young}
\date{Submitted 23rd September 2014, Revised 21st September 2017}

\begin{abstract} 
The set
\[
\G \df  \{(z+w,zw):|z| < 1,|w| < 1\} \subset \C^2
\]
has intriguing complex-geometric properties; it has a 3-parameter group of automorphisms, its distinguished boundary is a ruled surface homeomorphic to the M\"obius band and it has a special subvariety which is the only complex geodesic of $\G$ that is invariant under all automorphisms.  We exploit the geometry of $\G$ to develop an explicit and detailed structure theory for the rational maps from the unit disc to the closure $\Ga$ of $\G$ that map the boundary of the disc to the distinguished boundary of $\Ga$.
\end{abstract}



\thanks{The first author was partially supported by National Science Foundation Grants  DMS 1361720 and 1665260. The second and third authors were partially Partially supported by the Engineering and Physical Sciences grant 
 EP/N03242X/1. The collaboration was partially supported by London Mathematical Society Grant 41219 and by Newcastle University.}

\maketitle

\section{Introduction}\label{introJuly2014}
Recall that a {\em rational inner function} is a rational map $h$ from the open unit disc $\D$
in the complex plane $\C$ to its closure $\D^-$ with the property that $h$ maps the unit circle $\T$ into itself.  A basic
unsolved problem in $H^\infty$ control theory led us to consider ``inner mappings" from $\D$ to certain domains in $\C^d$ with $d>1$.  For example, a special case of the problem of robust stabilization under structured uncertainty, or the  $\mu$-synthesis problem \cite{Do,dullerud,AY99},  leads naturally to a class of domains of which a typical member is
\be\label{defG}
\G\df\{ (z+w,zw): z,w \in\D\} \subset \C^2,
\ee
the {\em open symmetrized bidisc}.  It would be useful for the control community if one could develop a theory of analytic maps from $\D$ to $\G$ and its generalizations  parallel to the rich function theory of $\D$, created by such masters as Stieltjes, F. and M. Riesz, Carath\'eodory, Herglotz, Pick and Nevanlinna, which has been so effective in circuits and systems engineering and in statistical prediction theory, among other applications.  The appropriate analog of rational inner functions is likely to play a significant role in such a theory.  A number of recent papers \cite{KZ2014,costara,ALY12,ALY13_2,ogle} and the present one provide evidence that a detailed analysis is indeed possible.
 
We denote the closure of $\G$ by $\Ga$ and we define a {\em rational $\Ga$-inner function} to be a rational analytic map $h:\D\to \Ga$ with the property that $h$ maps $\T$ into the distinguished boundary $b\Ga$ of $\Ga$.   Here, $b\Ga$  is the smallest closed subset of $\Ga$ on which every continuous function on $\Ga$ that is analytic in $\G$ attains its maximum modulus. 
Rational $\Ga$-inner functions have many similarities with rational inner functions, but also have some striking differences, which stem from the fact that $\G$ has a more subtle complex geometry than other well-studied domains in $\C^2$, such as the bidisc $\D^2$ and the Euclidean ball $\mathbb{B}_2$.  

Here are three points of difference between $\D^2$ and $\mathbb{B}_2$ on the one hand and $\G$ on the other.
Firstly, whereas $\D^2$ and  $\mathbb{B}_2$ are homogeneous (so that the holomorphic automorphisms of these domains act transitively),
$\G$ is inhomogeneous.  The orbit of $(0,0)$ under the group of automorphisms of $\G$ is the intersection of $\G$ with the
variety 
\be\label{eq1.30}
\royal\df \set{(2z,z^2)}{z\in\C} \notag
\ee
which we call the {\em royal variety}.   $\royal\cap \G$ is a complex geodesic in $\G$ and is the only complex geodesic that is left invariant by all automorphisms of $\G$.
The variety $\royal$ plays a central role in the function theory of $\G$.

Secondly, the distinguished boundary of $\G$ differs markedly in its topological properties from those of $\D^2$ and $\mathbb{B}_2$.
The distinguished boundaries of $\D^2$ and $\mathbb{B}_2$ are the 2-torus and the $3$-sphere respectively, which are smooth manifolds without boundary. In contrast, the distinguished boundary $b\Ga$ of $\G$
is homeomorphic to a M\"obius band together with its boundary, which is a circle.  For a rational $\Ga$-inner function $h$ the curve $h(\e^{it}), 0\leq t<2\pi$, lies in $b\Ga$ and may or may not touch the edge, with consequences for the algebraic and geometric properties of $h$.

Thirdly, neither the 2-torus nor the $3$-sphere contains any line segments, and therefore strict convex combinations of their inner functions can {\em never} be inner.  $b\Ga$, on the other hand, is a ruled surface, and so it {\em can} happen that a strict convex combination of rational $\Ga$-inner functions is a function of the same type. There is thus a notion of non-extremal $\Ga$-inner function.

These three geometric properties of $\G$ lead to dramatic phenomena in the theory of rational $\Ga$-inner functions  that have no analog for classical inner functions.     The first of our three main theorems, a corollary of Theorem \ref{prop3.10}, is as follows.
\begin{theorem}\label{prop1.20}
If $h$ is a  nonconstant  rational $\gaminn$ function then either $h(\d^-) = \royal\cap \Ga$ or $h(\d^-)$ meets $\royal$ exactly $\deg(h)$ times.
\end{theorem}
Here $\deg(h)$ is the degree of $h$, defined in a natural way by means of fundamental groups (Section \ref{sec3}). 
In Proposition \ref{degh} we show that, for any rational $\Gamma$-inner function $h= (s,p)$, $\deg (h)$ is equal to the degree $\deg (p)$ (in the usual sense) of the finite Blaschke product $p$.
The precise way of counting the number of times that $h(\d^-)$ meets $\royal$ is also described in Section \ref{sec3}.  We call the points $\la\in\d^-$ such that $h(\la)\in\royal$ the {\em royal nodes} of $h$ and, for such $\la$, we call $h(\la)$ a {\em royal point } of $h$.

The second main result describes the construction of rational $\Ga$-inner functions of prescribed degree from certain interpolation data; one can think of this result as an analog of the expression for a finite Blaschke product in terms of its zeros.   Concretely, the following result is a corollary of Theorem \ref{thm4.10}.
\begin{theorem}\label{thm1.15}
Let $n$ be a positive integer and suppose given 
\begin{enumerate}
\item  points $\alpha_1, \alpha_2,\ldots,\alpha_{k_0} \in \d$ and $\tau_1, \tau_2,\ldots,\tau_{k_1} \in \t$, where $2k_0 +k_1 = n$;
\item points $\si_1,\dots,\si_n$  in $\d^-$ distinct from $\tau_1,\ldots,\tau_{k_1}$.
\end{enumerate}
There exists a rational $\gaminn$ function $h=(s,p)$ of degree $n$ such that the zeros of $s$ are 
$ \alpha_1, \alpha_2,\ldots,\alpha_{k_0},\tau_1, \tau_2,\ldots,\tau_{k_1}$ and the royal nodes of $h$ are $\si_1,\dots,\si_n$.
\end{theorem}
The proof is constructive: it gives a prescription for the construction of a $2$-parameter family of such functions $h$, subject to the computation of  Fej\'er-Riesz factorizations of certain non-negative functions on the circle.

The third main result of the paper is as follows.
\begin{theorem}\label{2k>n}
A $\Ga$-inner function of degree $n$ having $k$ royal nodes in $\T$, counted with multiplicity, is an extreme point of the set of rational $\Ga$-inner functions if and only if $2k>n$.
\end{theorem}

Thus a $\Ga$-inner function $h$ is an extreme point of the set of rational $\Ga$-inner functions if and only if the curve $h(\e^{it})$ touches the edge of the M\"obius band $b\Ga$ more than $\half \deg(h)$ times, counted with multiplicity.


\section{Properties of $\Ga$ and $\Gamma$-inner functions}\label{sec2}

We shall often use the co-ordinates $(s,p)$ for points in the symmetrized bidisc $\Ga$, chosen to suggest `sum' and `product'.
The following results afford useful criteria for membership of $\gam$, $\partial \Gamma$ and $\M$ \cite{AY04}.
\begin{proposition}\label{prop2.10}
Let $(s,p) \in \ctwo$.  The point $(s,p)$ lies in $ \gam$ if and only if
\be\label{eq2.30}
|s| \le 2 \; \text{and} \; |s-\overline{s}p| \le 1-|p|^2. \notag
\ee
The point $(s,p)$ lies in $\M$ if and only if
\be\label{eq2.40}
|s| \le 2,\ |p| =1,\; \text{and} \ s-\overline{s}p = 0. \notag
\ee
The point $(s,p)$ lies in  $\partial \Gamma$ if and only if
\be\label{eq2.50}
|s|\le 2 ,\ \text{and}\ |s- \overline{s} p| =  1 - |p|^2, \notag
\ee
which is so if and only if
there exist $z\in\T \ \text{and}\ w \in \D^- \ \text{such that}\ s=z+w, \ p=zw$.
\end{proposition}

For any domain $G$ and any set $X\subset \D^d$ we shall denote by $\hol(G,X)$ the set of holomorphic maps from $G$ to $X$.  A {\em $\gam$-inner function } is defined to be a function $h\in\hol(\D,\gam)$ such that
\[
\lim_{r\to 1-} h(r\la) \in \M
\]
for almost all $\la\in\t$ with respect to Lebesgue measure.  Fatou's theorem implies that the radial limit above exists for almost all $\la\in\T$.  In this paper we focus on {\em rational  $\gam$-inner functions } $h$, that is, $h$ is $\gam$-inner and $s$ and $p$ are rational.
 
 Proposition \ref{prop2.10} implies that if  $h=(s,p) \in \hol(\D,\ctwo)$ is rational  then $h$ is $\gaminn$ if and only if $p$ is inner, $|s|$ is bounded by 2 on $\d$ and $s(\tau)-\overline{s(\tau)}p(\tau) = 0$ for almost all $\tau \in \t$ with respect to Lebesgue measure. The following proposition captures the algebra of this special case.

  For any positive integer  $n$ and  polynomial $f$ of degree less than or equal to $n$, we define the polynomial ${f}^{\sim n}$ by the formula,
\be\label{eq2.60}
{f}^{\sim n}(\lambda)= \lambda^n \overline{f(1/\overline{\lambda})}. \notag
\ee
Recall that $p$ is a rational inner function on $\d$ of degree $n$ (that is, a Blaschke product with $n$ factors) if and only if there exists a polynomial $D$ of degree less than or equal to $n$ such that $D(\lambda) \ne 0$ for all $\lambda \in \d^-$ and
\be\label{eq2.70}
p(\lambda) = \frac{D^{\sim n}(\lambda)}{D(\lambda)}. \notag
\ee
As an analog of this description of rational inner functions on $\d$ we have the following description of rational $\gaminn$ functions.
\begin{proposition}\label{prop2.20}
If $h=(s,p)$ is a rational $\gaminn$ function of degree $n$ then there exist polynomials $E$ and $D$ such that
\begin{enumerate}
\item[\rm (i)] $\ \deg(E),\deg(D) \le n, $
\item[\rm(ii)]$\ E^{\sim n}=E, $
\item[\rm(iii)]$\ D(\lambda) \ne 0\ on\ \d^-$,
\item[\rm(iv)]$\ |E(\lambda)| \le 2|D(\lambda)|\ on\ \d^-$,
\item[\rm(v)]$\ s=\frac{E}{D}$ on $\d^-$, and
\item[\rm(vi)]$\ p=\frac{D^{\sim n}}{D}$ on $\d^-$. 
\end{enumerate}
Furthermore, $E_1$ and $D_1$ is a second pair of polynomials satisfying {\rm (i)-(vi)} if and only if there exists a nonzero $t \in \r$ such that $E_1=tE$ and $D_1=tD$.

\noindent Conversely, if $E$ and $D$ are polynomials satisfying {\rm (i)}, {\rm (ii)}, {\rm(iv)} above, $ D(\lambda) \neq 0 $ on
$\D$,
and $s$ and $p$ are defined by {\rm (v)} and {\rm(vi)}, then $h=(s,p)$ is a rational $\gaminn$ function of degree  less than or equal to $n$.
\end{proposition}
\begin{proof} Let $h=(s,p)$ be a rational $\gaminn$ function of degree $n$. By  \cite[Corollary 6.10]{ALY12},
 $s$ and $p$ can be written as ratios of polynomials with the same denominators. More precisely, let
\begin{equation} \label{formp}
p(\lambda) = c \frac{\lambda^k D_p^{\sim (n-k)}(\lambda)}{D_p(\lambda)}
\end{equation}
where $|c|=1$, $0\le k \le n$ and $D_p$ is a polynomial of degree $n-k$ such that $D_p(0)=1$.  Then $s$ is expressible in the form 
\begin{equation} \label{forms}
s(\lambda) = \frac{\lambda^{\ell} N_s(\lambda)}{D_p(\lambda)} 
\end{equation}
where $0 \le \ell \le \frac{1}{2}n= \frac{1}{2} \deg(p)$,  and $N_s$ is a polynomial of degree $ \deg(p) - 2 \ell$ such that $N_s(0) \neq 0$. Moreover, for all $\la \in \C$, 
\begin{equation} \label{form-s-p}
\lambda^{\ell} N_s(\lambda)=
 c \lambda^{n-\ell}  \overline{N_s(1/\bar{\lambda})}.
\end{equation}

Note that $p$ is inner, and so one can choose $D_p$ such that
$ D_p(\lambda) \neq 0$ on $ \d^-$.
Take $D = \omega D_p$, where $\omega \in \T$ such that $\omega^2 c=1$, and $E(\lambda) = \omega \lambda^\ell N_s(\lambda)$.
It is easy to see that $\deg(E) \le n$ and $\deg(D) \le n$.
One can check that $E^{\sim n}=E, $
$ p=\frac{D^{\sim n}}{D}$, $s =\frac{E}{D}$. Since $|s| \le 2$ on $\d^-$ we have $|E(\lambda)| \le 2|D(\lambda)|$ on $\d^-$.

If there exists a nonzero $t \in \r$ such that $E_1=tE$ and $D_1=tD$ then it is clear that the polynomials $E_1$ and $D_1$ satisfy {\rm (i)-(vi)}. 
Conversely, let $E_1$ and $D_1$ be a second pair of polynomials satisfying {\rm (i)-(vi)}. Therefore
\begin{equation}\label{E=E1}
s=\frac{E}{D} = \frac{E_1}{D_1}\; \text {on } \; \d^-.
\end{equation} 
and
\begin{equation}\label{D=D1}
p=\frac{D^{\sim n}}{D}=\frac{D_1^{\sim n}}{D_1}\; \text {on } \; \d^-.
\end{equation} 
Let 
$$D(\lambda) = a_0 +a_1 \lambda + \dots + a_k \lambda^k$$
for some $k \le n$ and $a_0\neq 0$.
Then 
\begin{align*}
D^{\sim n}(\lambda) & =\lambda^n \overline{D(1/\overline{\lambda})}\\
 &= \overline{a_0}\lambda^n +\overline{a_1} \lambda^{n-1} + \dots + \overline{a_k} \lambda^{n-k}.
\end{align*}
Therefore, for all $\lambda \in \d^-$,
\begin{align*}
p(\lambda)& =\frac{D^{\sim n}(\lambda)}{D(\lambda)}\\
& =\frac{\lambda^{n-k}(\overline{a_0}\lambda^k +\overline{a_1} \lambda^{k-1} + \dots + \overline{a_k})}{a_0 +a_1 \lambda + \dots + a_k \lambda^k}.
\end{align*}
Thus $p$ has $k$ poles in $\C$, counting multiplicity, has a zero of multiplicity $n-k$ at $0$, and has degree $n$. Hence
$n, k$ and the poles of $p$ in $\{z \in \C: |z| >1 \}$ are determined by $p$. Therefore,
$D$ and $D_1$ have the same degree $k$ and the same finite number of zeros in $\{z \in \C: |z| >1 \}$, counting multiplicity.
Thus there exists $t \in \C$ such that $t \neq 0$ and $D_1= t D$.
By the equality \eqref{D=D1},
$$\frac{D^{\sim n}}{D}=\frac{D_1^{\sim n}}{D_1}=\frac{\bar{t}D^{\sim n}}{tD}.$$
Thus $t=\bar{t}$. By the  equality \eqref{E=E1}, $E_1=tE$.
Therefore there exists a nonzero $t \in \r$ such that $E_1=tE$ and $D_1=tD$.\\

Let us prove the converse result. Let
 $E$ and $D$ be polynomials satisfying {\rm (i)}, {\rm (ii)}, {\rm(iv)} above, $\ D(\lambda) \neq 0$ on $\D$, and let $s$ and $p$ be defined by {\rm (v)} and {\rm(vi)}.

Suppose $D$ has no zero on $\T$. Then $D$ and $D^{\sim n}$ have no common factor. Hence $p= \frac{D^{\sim n}}{D}$ is inner and
$$\deg(p)= \deg\left(\frac{D^{\sim n}}{D}\right) = \max\{\deg(D),\deg(D^{\sim n})\}=n.$$

If $D$ has  zeros $\sigma_1,  \dots, \sigma_\ell$ on $\T$
then $D$ and $D^{\sim n}$ have the common factor 
\newline $\prod_{i=1}^\ell(\lambda - \sigma_i)$. Thus $p= \frac{D^{\sim n}}{D}$ is inner and
$$\deg(p)= \deg\left(\frac{D^{\sim n}}{D}\right) \le n-\ell.$$
By {\rm(iv)}, $|s|$ is bounded by 2 on $\d^-$.
Let us show that $\overline{s}p= s$ almost everywhere on $\T$. That is,
$$
\overline{s}p  =\frac{\overline{E}}{\overline{D}}\frac{D^{\sim n}}{D}= \frac{{E}}{{D}}=s\;\text{almost everywhere on} \; \T.$$
For all $\lambda \in \t \setminus \{\sigma_1,  \dots, 
\sigma_\ell \}$, we have
\begin{align*}
 E(\lambda) = E^{\sim n}(\lambda) &\iff E(\lambda) = \lambda^n \overline{E(\lambda)}\\
&\iff  \overline{E(\lambda)}\lambda^n \overline{D(\lambda)}=\overline{D(\lambda)}E(\lambda)\\
&\iff  \overline{E(\lambda)} D^{\sim n}(\lambda) = \overline{D(\lambda)}E(\lambda)\\
&\iff 
\frac{\overline{E}(\lambda)}{\overline{D}(\lambda)}\frac{D^{\sim n}(\lambda)}{D(\lambda)}= \frac{{E}(\lambda)}{{D}(\lambda)}.
\end{align*}
Thus $s(\lambda)=\overline{s(\lambda)}p(\lambda)$ for all but finitely many $\la\in\T$.
By Proposition \ref{prop2.10}, $h=(s,p)$ is a rational $\gaminn$ function of $\deg(h) \le n-\ell$.
\end{proof}
\begin{remark} \label{dimE} {\rm 
For fixed $D$ of degree $n$, the set of polynomials $E$ satisfying (i) and (ii) of the previous proposition is a real vector space of dimension $n+1$.
}
\end{remark}


\section{Royal nodes and the royal polynomial} \label{sec3}

The royal variety 
\[
\royal= \{(2\la,\la^2):\la\in\C\}=\{(s,p)\in\C^2 : s^2=4p\} 
\]
plays a prominent role in the geometry and function theory of $\G$.  Its intersection with $\G$ is a complex geodesic of $\G$,
it is invariant under all automorphisms of $\G$, and it is the {\em only} complex geodesic of $\G$ with this invariance property
\cite[Lemma 4.3]{AY08}.   In this section we describe rational $\Ga$-inner functions in terms of their intersections with $\royal$.


 Let us clarify the notion of the degree of a rational
$\Gamma$-inner function $h$. 
\begin{definition}
 The degree $\deg(h)$ of a rational $\Ga$-inner function $h$ is defined to be  $ h_*(1)$, where 
$h_* : \mathbb{Z}= \pi_1(\T) \to \pi_1(\M)$ is the homomorphism of fundamental groups induced by $h$ when it is regarded as a continuous map from $\T$ to $\M$.
\end{definition}
This definition only defines $\deg(h)$ as an integer up to multiplication by $\pm 1$, as is shown by the following lemma.  We shall adopt the convention that $\deg(h)$ is a non-negative integer.
\begin{lemma} $\M$ is homotopic to the circle $\T$ and $\pi_1(\M) = \Z$.
\end{lemma}
\begin{proof} This is a consequence of the fact that if $(s,p) \in \M$ and $0\le t \le 1$ then also $(ts,p) \in \M$. Thus the maps
\[
f: \M \to \T, \;\; g: \T \to \M
\]
given by 
\[
f(s,p)=p, \;\; g(z)=(0,z)
\]
satisfy $(g \circ f)(s,p) =(0,p)$ and $ f \circ g ={\rm id}_{\T}$.
Thus $g \circ f \sim {\rm id}_{\M}$ since
\[
(t,s,p) \mapsto (ts,p)
\]
defines a homotopy between $g \circ f$ and ${\rm id}_{\M}$.
It follows that $\pi_1(\M) = \pi_1(\T)= \Z$.
\end{proof}

Thus $\pi_1(\M) = \Z$, and so $\deg(h)$ is an integer, which we can take to be nonnegative. 

\begin{proposition}\label{degh}
For any rational $\Gamma$-inner function $h= (s,p)$, $\deg (h)$ is the degree $\deg (p)$ (in the usual sense) of the finite Blaschke product $p$. 
\end{proposition}
\begin{proof}
Indeed, if 
\[
h_t(\lambda) = (t s(\lambda), p(\lambda))\;\; \text{for} \;\; t \in [0,1],
\]
then since $(ts(\lambda),p(\lambda)) \in \M$ for $\lambda  \in \T$, $h_t$ is a homotopy between $h=h_1$ and $h_0= (0,p(\lambda))$.
It follows that the homomorphism $h_* :\pi_1(\T)= \mathbb{Z} \to \pi_1(\M)=\Z$ coincides with $(h_0)_*$, and it is easy to see that $(0,p)_*(1)$ is the degree of $p$.
\end{proof}

Clearly the degree is a homotopy invariant in the class of rational $\Ga$-inner functions with the topology of uniform convergence on $\d^-$.  Indeed, the degree provides a complete homotopy invariant, in view of the fact that a Blaschke product of degree $n\geq 0$ is homotopic in the class of all finite Blaschke products to the Blaschke product $\la \mapsto  \la^n$.

An interesting and surprising fact is that if $h$ has degree $n$, then there are always exactly $n$ royal nodes of $h$ (Theorem \ref{prop3.10} below).
This fact, however, requires that the royal nodes be counted with the proper multiplicity. Consider a rational $\gaminn$ function $h=(s,p)$ of degree $n$, let $E$ and $D$ be as in Proposition \ref{prop2.20}, and let $R$ be the polynomial defined by
\be\label{eq3.25}
R(\lambda) = 4D(\lambda)D^{\sim n}(\lambda)-E(\lambda)^2.
\ee
We see using (iv) and (v) that
\be\label{eq3.30}
R(\lambda) = D(\lambda)^2(4p(\lambda) - s(\lambda)^2) \notag
\ee
for all $\lambda \in \d^-$. Hence, Condition (iii) in Proposition \ref{prop2.20} implies that the royal nodes of $h$ exactly correspond to the zeros of $R$ in $\d^-$. For this reason we refer to $R$ as the \emph{royal polynomial of $h$}. As Proposition \ref{prop2.20} states that $E$ and $D$ are only determined by $s$ and $p$ up to multiplication by a nonzero real scalar, ``the" royal polynomial is only determined by $s$ and $p$ up to multiplication by a positive scalar. We shall tolerate this slight abuse of language.

A special class of rational $\Ga$-inner functions consists of functions $h$ such that $h(\d)\subset \mathcal R$. These are precisely the rational $\Ga$-inner functions whose royal polynomials are identically zero. Such functions have the form $h=(2f,f^2)$ for some finite Blaschke product $f$, and so their theory reduces to that of finite Blaschke products.  Our concern is primarily with functions that are not of the form $(2f,f^2)$ (and so have nonzero royal polynomials).  For completeness, we shall define the degree of the zero polynomial to be $-\infty$.

\begin{defin}\label{def3.20}
 We say that a polynomial  $f$ is $n$-symmetric if $\deg(f) \le n$ and $f^{\sim n} = f$. For any set $E \subset {\mathbb C}$, ${\rm ord}_E (f)$ will denote the number of zeros of $f$ in $E$, counted with multiplicity, and ${\rm ord}_0 (f)$  will mean the same as ${\rm ord}_{\{0 \}} (f)$.

\end{defin}
Two simple facts are that, for any  $n$-symmetric nonzero polynomial $f$,
\be\label{eq3.32}
\deg(f)=n-{\rm ord}_{0} (f)
\ee
and
\be\label{eq3.34}
{\rm ord}_{0} (f) + 2{\rm ord}_{\D  \setminus\{0 \} } (f) + {\rm ord}_{\T} (f) = \deg(f).
\ee

\begin{proposition}\label{prop3.5}
Let $h$ be a rational $\gaminn$ function of degree $n$ and let $R$ be the royal polynomial of $h$ as defined by equation \eqref{eq3.25}. Then $R$ is $2n$-symmetric and the zeros of $R$ that lie on $\t$ have either even
or infinite order.
\end{proposition}
\begin{proof}
Clearly $D D^{\sim n}$ is 2$n$-symmetric. Condition (ii) in Proposition \ref{prop2.20} implies that $E^2$ is $2n$-symmetric. Hence, $R$ is $2n$-symmetric.

In the case $h(\d^-) \subset \royal\cap \Ga$, the royal polynomial  
$R $ is identically zero. Hence the zeros of $R$ on $\t$ have infinite order.

In other cases, note that if $\lambda \in \t$, then $D^{\sim n}(\lambda) = \lambda^n \overline{D(\lambda)}$ and $E(\lambda) = E^{\sim n}(\lambda) = \lambda^n \overline{E(\lambda)}$. Hence, if $\lambda \in \t$,
\be\label{eq3.36}
\lambda^{-n}R(\lambda) = \lambda^{-n}(4D(\lambda)D^{\sim n}(\lambda)-E(\lambda)^2)
=4|D(\lambda)|^2-|E(\lambda)|^2.
\ee
As Condition (iv) in Proposition \ref{prop2.20} guarantees that $4|D|^2-|E|^2 \ge 0$ on $\t$, it follows that the zeros of $R$ that lie in $\t$ have even order.
\end{proof}
Proposition \ref{prop3.5} shows that the rule for counting royal nodes introduced in the following definition always produces an integer.

\begin{definition}\label{def3.30}
Let $h$ be a rational $\gaminn$ function 
such that  $h(\d^-) \not\subseteq \royal\cap \Ga$ 
and let  $R$ be the royal polynomial of $h$.  If 
 $\sigma$ is a zero of $R$ of order $\ell$, we define the \emph{multiplicity $\#\sigma$} of $\sigma$ (as a royal node of $h$) by 
\[
 \#\sigma \quad = \quad \twopartdef{ \ell} {\mbox{ if }\sigma \in \d}{\half \ell} {\mbox{ if }\sigma \in \t.}
\]
We define the {\em type} of $h$ to be the ordered pair $(n,k)$ where $n$ is the sum of the multiplicities of the royal nodes of $h$ that lie in $\d^-$ and  $k$ is the sum of the multiplicities of the royal nodes of $h$ that lie in $\t$. We define $\rnk$ to be the collection of rational $\gaminn$ functions $h$ of type $(n,k)$.

\end{definition}

 $\rnk$ makes sense for integers $n,k$ such that $0\le k \le n$.
The following example of rational $\gaminn$ functions from $
\mathcal{R}^{n,k}$ for even $n \ge 2$ can be found in \cite[Proposition 12.1]{ALY12}.

\begin{example}{\rm
 For all $\nu \ge 0$ and $0<r<1$,  the function 
\begin{equation}\label{h_nu}
h_{\nu}(\la) =\left( 2(1-r) \frac{\la^{\nu+1}}{1+ r\la^{2\nu+1}},
\frac{\la(\la^{2\nu+1}+r)}{1+ r\la^{2\nu+1}} \right), \;\la \in \D,
\end{equation}
belongs to $\mathcal{R}^{2 \nu +2,2 \nu +1}$.
The royal nodes of $h_\nu$ that lie in $\t$, being the points at which $|s|=2$, are the $(2\nu +1)$th roots of $-1$, that is, 
\[
\omega_j = e^{\ii \pi (2j+1)/(2\nu+1)}, \; j= 0, \dots, 2 \nu.
\]
They are all of multiplicity $1$.
There is a simple royal node at $0$.
}
\end{example}

The following proposition clarifies the statement in the introduction that, with correct counting, a rational $\gaminn$ function $h$ has exactly $\deg(h)$ royal nodes. 
\begin{theorem}\label{prop3.10}
If $h \in \rnk$ 
is nonconstant then $n=\deg(h)$.
\end{theorem}
\begin{proof}
Let $R$ be the royal polynomial of $h$.
By assumption $h \in \rnk$, hence $n \ge 1$ and $h(\d^-) \not\subseteq \royal\cap \Ga$. Thus $R$ is not identically zero.
 As Proposition \ref{prop3.5} implies that $R$ is $2\deg(h)$-symmetric, it follows from equations \eqref{eq3.32} and \eqref{eq3.34} that
\be\label{eq3.40}
\deg(R)= 2\deg(h)-{\rm ord}_{0} (R)\notag
\ee
and
\be\label{eq3.50}
{\rm ord}_{0} (R) + 2{\rm ord}_{\D  \setminus\{0 \} } (R) + {\rm ord}_{\T} (R) = \deg(R),\notag
\ee
which imply that
\be\label{eq3.60}
2{\rm ord}_{0} (R) + 2{\rm ord}_{\D  \setminus\{0 \} } (R) + {\rm ord}_{\T} (R) = 2\deg(h).
\ee
Hence
$$
n ={\rm ord}_{0} (R) + {\rm ord}_{\D  \setminus\{0 \} } (R) + \half{\rm ord}_{\T} (R) = \deg(h).
$$
\end{proof}

Note that the number of royal nodes of $h$ is equal to 
$$
{\rm ord}_{0} (R) + {\rm ord}_{\D  \setminus\{0 \} } (R) + \half{\rm ord}_{\T} (R).
$$
Therefore, as a corollary of Theorem \ref{prop3.10} we obtain   Theorem \ref{prop1.20}, that is, for a  nonconstant  rational
 $\gaminn$ function  $h$,  either $h(\d^-) = \royal\cap \Ga$ or $h(\d^-)$ meets $\royal$ exactly $\deg(h)$ times.

\begin{example}\label{boring} {\rm Let $\phi$ and $\psi$ be inner functions on $\D$. Then
\begin{equation}\label{phipsi}
h = (\phi + \psi, \phi \psi)
\end{equation}
is $\Gamma$-inner. In particular, $h = (2 \phi, \phi^2)$ is 
$\Gamma$-inner; this example has the property that $h(\D)$ lies in the royal variety $\royal$.
}
\end{example}

\section{Prescribing the royal nodes of $h$ and zeros of $s$} \label{royalnodes}

In this section we shall show how to construct a rational $\gaminn$ function $h=(s,p)$ with the royal nodes of $h$ and the zeros of $s$ prescribed. Recall from the previous section that if $R$ denotes the royal polynomial of $h$ then the royal nodes are the zeros of $R$ that lie in $\d^-$. 
In addition, if $\deg(h)=n$ then $R$ is $2n$-symmetric, has zeros of even order on $\t$ and, by equation \eqref{eq3.36}, satisfies $\lambda^{-n}R(\lambda) \ge 0$ for all $\lambda \in \t$. We formalize these properties in the following definition.
\begin{definition}\label{def4.10}
A nonzero polynomial $R$ is {\em $n$-balanced} if $\deg(R) \le 2n$, $R$ is $2n$-symmetric and $\lambda^{-n}R(\lambda) \ge 0$ for all $\lambda \in \t$.
\end{definition}

We have shown the following.
\begin{lemma}\label{royal_polyn} Let $h$ be a rational $\Ga$-inner function $h$ of degree $n$. Then the royal polynomial $R$ of $h$ is either $n$-balanced or identically zero.
\end{lemma}

To construct a rational $\gaminn$ function $h=(s,p)$ with the royal nodes of $h$ and the zeros of $s$ prescribed we require the following well-known result of Fej\'er and Riesz (\cite[Chapter 6]{roro85}, \cite[Section 53]{RN}).

\begin{lemma}\label{lem4.30}
If $f(\lambda)=\sum_{i=-n}^n a_i \lambda^i$ is a trigonometric polynomial of degree $n$ such that $f(\lambda) \ge 0$ for all $\lambda \in \t$ then there exists an analytic polynomial  $D(\lambda)=\sum_{i=0}^n b_i \lambda^i$ of degree $n$  such that $D$ is outer (that is, $D(\lambda) \ne 0$ for all $\lambda \in \d$) and
$$
f(\lambda)=|D(\lambda)|^2
$$
for all $\lambda \in \t$. 
\end{lemma}

Let us give brief proofs of the following elementary lemmas, which can be used to build $n$-balanced polynomials from their zero sets.
\begin{lemma}\label{lem4.10}
For $\sigma \in \d^-$, let the polynomial $Q_\sigma$ be defined by the formula
$$
Q_\sigma(\lambda)=(\lambda - \sigma)(1-\overline{\sigma}\lambda).
$$
Let $n$ be a positive integer and let $R$ be a nonzero polynomial. $R$ is $n$-balanced if and only if there exist points $\sigma_1, \sigma_2,\ldots,\sigma_n \in \d^-$  and $t_+ >0$ such that
\be\label{eq4.10}
R(\lambda) = t_+\prod_{j=1}^n Q_{\sigma_j}(\lambda),\;\quad \lambda \in \C.
\ee
\end{lemma}
Note that there may be repetitions among the $\si_j$.

\begin{proof} It is easy to show that $R$ is $n$-balanced 
if \eqref{eq4.10} holds.

Conversely, let $R$ be $n$-balanced; then $R$ is $2n$-symmetric. We have
$$
 R(\lambda) = r_0 + r_1 \lambda +  \dots + \overline{r_1} \lambda^{2n-1} +  \overline{r_0} \lambda^{2n} 
$$
and 
$$
 r_0 \overline{\lambda}^n + r_1 \overline{\lambda}^{n-1} +  \dots + \overline{r_1} \lambda^{n-1} +  \overline{r_0} \lambda^{n} \ge 0 \; \text{for all} \; \lambda \in \T.  
$$
By Lemma \ref{lem4.30} there exists an outer polynomial 
$$
P(\lambda) = a_0 + a_1 \lambda +  \dots +  a_n \lambda^{n} 
$$
such that
$$
\overline{\lambda}^{n} R(\lambda) = |P(\lambda)|^2 \; \text{for all} \; \lambda \in \T 
$$
and $P$ has all its zeros $\zeta_i$  in $\C \setminus \D$.
Say
$$
P(\lambda) = c (\lambda - \zeta_1)\dots (\lambda - \zeta_n).
$$
Therefore, for all $ \lambda \in \T$,
\begin{align*}
\overline{\lambda}^n R(\lambda) &= |P(\lambda)|^2 \\
 &= |c|^2 (\lambda - \zeta_1)(\overline{\lambda} - \overline{\zeta_1})\dots (\lambda - \zeta_n)(\overline{\lambda} - \overline{\zeta_n})\\
 &= |c|^2 \overline{\lambda}^n(\lambda - \zeta_1)
(1 - \overline{\zeta_1} \lambda)\dots (\lambda - \zeta_n)(1 - \overline{\zeta_n} \lambda) \\
 &= \overline{\lambda}^n (|c| |\zeta_1|\dots|\zeta_n|)^2 Q_{\sigma_1}(\lambda)\dots Q_{\sigma_n}(\lambda)
\end{align*}
where $\sigma_i = 1/\overline{\zeta_i} \in \d^-, \; i =1,\dots, n.$
Thus 
$$
R(\lambda) = t_+ Q_{\sigma_1}(\lambda)\dots Q_{\sigma_n}(\lambda) \; \text{for all} \; \lambda \in \C ,
$$
where $t_+ = (|c| |\zeta_1|\dots|\zeta_n|)^2$.
\end{proof}

The royal polynomial of a rational $\Ga$-inner function $h$ is determined by the royal nodes of $h$, with their multiplicities.

\begin{prop}\label{determined}
Let the royal nodes of a rational $\Ga$-inner function $h$ be $\si_1,\dots,\si_n$, with repetitions according to multiplicity of the nodes as described in Definition {\rm \ref{def3.30}}.
The royal polynomial $R$ of $h$, up to a positive multiple, is
\be\label{formR}
R(\la)= \prod_{j=1}^n Q_{\si_j}(\la).
\ee
\end{prop}
\begin{proof} By Lemma \ref{royal_polyn}, the royal polynomial $R$ of $h$  is $n$-balanced. By Lemma \ref{lem4.10}, there
exists $t_+ >0$ and $\zeta_1,\ldots,\zeta_n\in \d^-$ such that 
\[
R(\la)= t_+ \prod_{j=1}^n Q_{\zeta_j}(\la).
\]
Since the royal nodes of $h$ and their multiplicities are defined in terms of the zeros of $R$ in $\d^-$  and their multiplicities, it follows that the list $\zeta_1,\dots,\zeta_n$ coincides, up to a permutation, with the list $\si_1,\ldots,\si_n$.   Hence  $R$ is given, up to a positive multiple, by equation \eqref{formR}.
\end{proof}
Our next lemma summarizes an elementary procedure for building an $n$-symmetric polynomial from its roots. For $\tau = \e^{i\theta},0 \le \theta <2 \pi$, we define $L_\tau$ by
$$
L_\tau(\lambda) = i\e^{-i \frac{\theta}{2}}(\lambda - \tau).
$$
We note that $L_\tau$ is 1-symmetric:
\begin{align*}
L_\tau^{\sim 1}(\lambda)
&= \lambda \overline{L_\tau(1/\overline {\lambda})}\\
&=\lambda\overline{i\e^{-i\frac{\theta}{2}}(\lambda - \tau)}\\
&=-i\e^{i \frac{\theta}{2}}\overline{\tau}(\tau-\lambda)\\
&=L_\tau(\lambda),
\end{align*}
and that $L_\tau^2 = Q_\tau$:
\begin{align*}
L_\tau^2(\lambda)
&=  (i\e^{-i\frac{\theta}{2}}(\lambda - \tau))^2\\
&=-\overline{\tau}(\lambda - \tau)(\lambda - \tau)\\
&=(\lambda - \tau)(1-\overline{\tau}\lambda)\\
&=Q_\tau(\lambda).
\end{align*}
\begin{lemma}\label{lem4.20}
Let $n$ be a positive integer. A polynomial $E$ is $n$-symmetric if and only if there exist points $\alpha_1, \alpha_2,\ldots,\alpha_{k_0} \in \d$, points $\tau_1, \tau_2,\ldots,\tau_{k_1} \in \t$
 and $t \in \r$ such that
\begin{align}\label{eq4.20}
k_0 &= \ord_0(E) + \ord_{\D\setminus\{0\}}(E), \notag \\
k_1 &= \ord_\T(E),  \notag \\
2k_0 +k_1 &= n \quad \mbox{  and } 
\end{align}
\begin{align}\label{eq4.20_2}
E(\lambda) &= t\prod_{j=1}^{k_0} Q_{\alpha_j}(\lambda)
\prod_{j=1}^{k_1} L_{\tau_j}(\lambda).
\end{align}
\end{lemma}
\begin{proof} Let $E$ be $n$-symmetric, that is,
$$
\lambda^n \overline{E(1/\overline{\lambda})} = E(\lambda) \; \text{for all } \; \lambda \in \C \setminus \{0 \}.
$$
Let
$$
E(\lambda) = c (\lambda - \alpha_1)\dots (\lambda - \alpha_{k_0})(\lambda - \eta_1)\dots (\lambda - \eta_{k_1})(\lambda - \beta_1)\dots (\lambda - \beta_{k_2})\; \text{for all} \; \lambda \in \C,
$$
where $\alpha_1, \dots, \alpha_{k_0} \in \D$, $ \;\eta_1, \dots, \eta_{k_1} \in \T$ and $ \;\beta_1, \dots, \beta_{k_2} \in \C \setminus \d^-$.

Note that if $\alpha \in \D \setminus \{0 \}$ is a zero of $E$ then $1/\overline{\alpha} $ is also a zero of $E$ since
$$
\alpha^n \overline{E(1/\overline{\alpha})} = E(\alpha)=0.
$$
Hence, for some $\nu$, $\nu \le k_0$,
$$
E(\lambda) = c \lambda^\nu (\lambda - \alpha_{\nu +1})\dots (\lambda - \alpha_{k_0})(\lambda - \eta_1)\dots (\lambda - \eta_{k_1})(\lambda - 1/\overline{\alpha_{\nu +1}})\dots (\lambda - 1/\overline{\alpha_{k_0}})
$$
for all $ \lambda \in \C.$
Thus
\begin{align*}
E(\lambda) &= \omega_1 \lambda^\nu (\lambda - \alpha_{\nu +1})(1 - \overline{\alpha_{\nu +1}}\lambda)\dots (\lambda - \alpha_{k_0})(1-\overline{\alpha_{k_0}}\lambda) 
(\lambda - \eta_1)\dots (\lambda - \eta_{k_1})\\
	& = \omega_1 \omega_2
Q^\nu_{0}(\lambda)Q_{\alpha_{\nu +1}}(\lambda)\dots
Q_{\alpha_{k_0}}(\lambda) L_{\eta_1}(\lambda)\dots L_{\eta_{k_1}}(\lambda)
\end{align*}
for all $ \lambda \in \C$, 
where $\omega_1= c (-1/\overline{\alpha_{\nu +1}})\dots (-1/\overline{\alpha_{k_0}})$ and $\omega_2 = \overline{i\e^{-i \frac{\eta_1}{2}} \dots i\e^{-i \frac{\eta_{k_1}}{2}}}$.
Recall that  $E$ is $n$-symmetric, $L$ is $1$-symmetric and $Q$ is $2$-symmetric; therefore
\begin{align*}
E(\lambda) &= \omega_1 \omega_2
Q^\nu_{0}(\lambda)Q_{\alpha_{\nu +1}}(\lambda)\dots
Q_{\alpha_{k_0}}(\lambda) L_{\eta_1}(\lambda)\dots L_{\eta_{k_1}}(\lambda)\\
	&= \lambda^n \overline{E(1/\overline{\lambda})}\\
	&= \lambda^n \overline{\omega_1 \omega_2} \overline{Q^\nu_{0}(1/\overline{\lambda})}\overline{Q_{\alpha_{\nu +1}}(1/\overline{\lambda})}\dots \overline{Q_{\alpha_{k_0}}(1/\overline{\lambda})}\overline{L_{\eta_1}(1/\overline{\lambda})} \dots 
\overline{L_{\eta_{k_1}}(1/\overline{\lambda})} \\
	&= \lambda^{n-2k_0-k_1} \overline{\omega_1 \omega_2} 
Q^\nu_{0}(\lambda) Q_{\alpha_{\nu +1}}(\lambda) \dots
Q_{\alpha_{k_0}}(\lambda) L_{\eta_1}(\lambda)\dots L_{\eta_{k_1}}(\lambda)
\end{align*}
for all $\lambda \in \C \setminus \{0 \}.$
Hence, $n =2k_0 +k_1$ and $\omega_1 \omega_2 \in \R$.

The converse result is easy.
\end{proof}

Note that there may be repetitions in the lists $\al_1,\ldots,\al_{k_0}$ and $\tau_1,\ldots, \tau_{k_1}$ above.

\begin{remark}\label{nodes-zeros} \rm
If $h=(s,p)$ is a rational $\Ga$-inner function 
 then no zero of $s$ on $\t$ can be a royal node of $h$.  For if $s=0=s^2-4p$ then $p=0$, whereas $|p|=1$ at every point of $\t$ at which $p$ is defined, including every royal node of $h$ on $\t$.
\end{remark}

We can now elucidate Theorem \ref{thm1.15} on the existence of rational $\Ga$-inner functions of prescribed degree with a given nodal set and a  given zero set of $s$.  The following result not only asserts the existence of the desired function, but also describes how to construct all such functions.

\begin{theorem}\label{thm4.10}
Let $n$ be a positive integer and suppose 
  points $\alpha_1, \alpha_2,\ldots,\alpha_{k_0} \in \d$ and 
 $\tau_1, \tau_2,\ldots,\tau_{k_1} \in \t$ are given, where $2k_0 +k_1 = n$, and
 points $\si_1,\dots,\si_n$  in $\d^-$ distinct from $\tau_1,\ldots,\tau_{k_1}$.

There exists a rational $\gaminn$ function $h=(s,p)$ of degree $n$ such that 
\begin{enumerate}
\item the zeros of $s$, repeated according to multiplicity, are 
$\alpha_1, \alpha_2,\ldots,\alpha_{k_0}$ 
\newline and $\tau_1, \tau_2,\ldots,\tau_{k_1}$, 
\item the royal nodes of $h$ are $\si_1,\dots,\si_n$.  
\end{enumerate}

Such a function $h$ can be constructed as follows.
 Let $t_+>0$ and let $t \in \R \setminus \{0 \}$. Let $R$ and $E$ be defined by
\[
R(\la)= t_+ \prod_{j=1}^n (\la-\sigma_j)(1-\overline \si_j\lambda)
\]
and 
\[
E(\la)= t\prod_{j=1}^{k_0} (\la-\alpha_j)(1-\overline\al_j \lambda) \prod_{j=1}^{k_1} i\e^{-i\theta_j/2}(\la-\tau_j)
\]
where $\tau_j=\e^{i\theta_j}, \, 0\leq \theta_j < 2\pi$. 

\noindent {\rm (i)}  There exists an outer polynomial $D$ of degree at most $n$ such that
\be\label{eq4.30}
\lambda^{-n}R(\lambda) + |E(\lambda)|^2 = 4|D(\lambda)|^2
\ee
for all $\la\in\T$. 

\noindent {\rm (ii)} The function $h$ defined by
$$
h=(s, p) = \left(\frac{E}{D},\frac{D^{\sim n}}{D}\right)
$$
is a rational $\gaminn$ function such that $\deg(h)= n$ and conditions {\rm (1)} and {\rm (2)} hold.
The royal polynomial of $h$ is  $R$. 
\end{theorem}
\begin{proof}
(i) 
By Lemma \ref{lem4.10}, $R$ is $n$-balanced, and so
$\lambda^{-n}R(\lambda) \ge 0$ for all $\lambda \in \t$.
Therefore
$$
\lambda^{-n}R(\lambda) + |E(\lambda)|^2 \ge 0
$$
for all $\lambda \in \t$. By Lemma \ref{lem4.30}, there exists  an outer polynomial $D$ of degree at most $n$ such that the
equality \eqref{eq4.30} holds.

(ii) By Lemma \ref{lem4.20},  the  polynomial $E$ is $n$-symmetric.
Let $D$ be an  outer polynomial of degree at most $n$ such that
the equality \eqref{eq4.30} holds for all $\la\in\T$.
By hypothesis
$$
\set{\sigma_j}{1 \le j \le n} \cap \set{\tau_j}{1 \le j \le k_1} = \emptyset.
$$
Then $\lambda^{-n}R(\lambda)$ and  $ |E(\lambda)|^2 $ are non-negative trigonometric polynomials on $\T$ with no common zero.
Thus 
$$
\lambda^{-n}R(\lambda) + |E(\lambda)|^2 > 0 \; \text{on} \; \T.
$$
By the equality \eqref{eq4.30}, $D$ has no zero on $\T$, and so $D$ and $D^{\sim n}$ have no common factor. Hence 
$$\deg(p)= \deg\left(\frac{D^{\sim n}}{D}\right) = \max\{\deg(D),\deg(D^{\sim n})\}=n.$$

Since
$\lambda^{-n}R(\lambda) \ge 0$ for all $\lambda \in \t$,
$$
4|D(\lambda)|^2 =
\lambda^{-n}R(\lambda) + |E(\lambda)|^2 \ge |E(\lambda)|^2
$$
for all $\lambda \in \t$. Thus
$$
|E(\lambda)| \le 2 |D(\lambda)|
$$
for all $\lambda \in \t$. Since $\ D(\lambda) \neq 0$ on $ \d^-$, we have  
$$\left|\frac{E(\lambda)}{D(\lambda)} \right| \le 2\;\;\text{for all}\;\; \lambda \in \d^-.$$
By the converse of Proposition \ref{prop2.20}, $$
h(\lambda)=\left(\frac{E(\lambda)}{D(\lambda)},\frac{D^{\sim n}(\lambda)}{D(\lambda)}\right),
$$
is a rational $\gaminn$ function with  $\deg(h)= n$.

The royal polynomial of $h$ is defined in equation \eqref{eq3.30} by
$$R_h(\lambda) = 4D(\lambda)D^{\sim n}(\lambda)-E(\lambda)^2 .$$
By equation \eqref{eq3.36}, for all $\lambda \in \T$,
$$\lambda^{-n}R_h(\lambda) = 4D(\lambda)\overline{D(\lambda)}-E(\lambda)\overline{E(\lambda)} .$$
By equation \eqref{eq4.30}, for  all $\lambda \in \T$,
$$\lambda^{-n}R_h(\lambda) =\lambda^{-n}R(\lambda).$$
Thus the royal polynomial of $h$ is equal to $R$.
\end{proof}

For large $n$ the task of finding an outer polynomial $D$ satisfying equation \eqref{eq4.30} cannot be solved algebraically.  It can, however, be efficiently solved numerically; engineers call this the problem of {\em spectral factorization}, and they have elaborated effective algorithms for it -- see for example \cite{mathworks}.

The solution $D$ of the spectral factorization \eqref{eq4.30} is only determined up to multiplication by a unimodular constant $\bar\omega$.  If we replace $D$ by $\bar\omega D$ then we obtain a new solution
\[
h=\left(\omega\frac{E}{D}, \omega^2\frac{D^{\sim n}}{D}\right).
\]
It appears at first sight that the construction in Theorem {\rm \ref{thm4.10}} gives us a $3$-parameter family of rational $\Ga$-inner functions with prescribed royal nodes and prescribed zero set of $s$, since we may choose $t_+,\, t$ and $\omega$ independently.  However, the choice of $1, \, t/\sqrt{t_+}, \, D/\sqrt{t_+}$ and $\omega$ leads to the same $h$ as $t_+,\, t, \, D$ and $\omega$. The following statement tells us that the construction yields {\em all} solutions of the problem, and so the family of functions $h$ with the required properties is indeed a $2$-parameter family.

\begin{proposition}\label{general_sol} Let $h=(s,p)$ be a rational $\gaminn$ function of degree $n$ such that 
\begin{enumerate}
\item the zeros of $s$, repeated according to multiplicity, are 
$\alpha_1, \alpha_2,\ldots,\alpha_{k_0}\in\d$, \newline $\tau_1, \tau_2,\ldots,\tau_{k_1}\in\t$,  where $2k_0 +k_1 = n$, and 
\item the royal nodes of $h$ are $\si_1,\dots,\si_n$.  
\end{enumerate}
There exists some choice of $t_+ >0$, $t \in \R \setminus \{0 \}$  and $\omega \in \T$ such that the recipe 
in Theorem {\rm \ref{thm4.10}} with these choices produces the function $h$.
\end{proposition}
\begin{proof} By Proposition \ref{prop2.20}, 
 there exist polynomials $E_1$ and $D_1$ such that 
$\deg(E_1),\deg(D_1) \le n$, $E_1$ is $n$-symmetric, $D_1(\lambda) \neq 0$ on $ \d^-$,
and
\[
s=\frac{E_1}{D_1}\; \text{ and} \;
p=\frac{D_1^{\sim n}}{D_1}\; \text{on} \; \d^-. 
\]
By hypothesis,
the zeros of $s$, repeated according to multiplicity, are 
$\alpha_1, \alpha_2,\ldots,\alpha_{k_0}$, $\tau_1, \tau_2,\ldots,\tau_{k_1}$,  where $2k_0 +k_1 = n$.
Since $E_1$ is $n$-symmetric, by  Lemma \ref{lem4.20}, there exists $t \in \R \setminus \{0 \}$  such that
\[
E_1(\la)= t\prod_{j=1}^{k_0} (\la-\alpha_j)(1-\overline\al_j \lambda) \prod_{j=1}^{k_1} i\e^{-i\theta_j/2}(\la-\tau_j).
\]
The royal nodes of $h$ are assumed to be $\si_1,\dots,\si_n$.
By Proposition \ref{determined}, for the royal polynomial $R_1$ of $h$,  there
exists $t_+ >0$ such that 
\[
R_1(\la)= t_+ \prod_{j=1}^n Q_{\si_j}(\la).
\]
By the equality \eqref{eq3.36}, for the royal polynomial $R_1$ of $h$, we have 
\[
\lambda^{-n}R_1(\lambda) = \lambda^{-n}(4D_1(\lambda)D_1^{\sim n}(\lambda)-E_1(\lambda)^2)
=4|D_1(\lambda)|^2-|E_1(\lambda)|^2,
\]
for $\lambda \in \t$. 
Since $E_1$ and $R_1$ coincide with $E$ and $R$ in the construction of Theorem \ref{thm4.10}, for a suitable choice of
of $t_+ >0$ and $t \in \R \setminus \{0 \}$, $D_1$ is a permissible choice for $\omega D$ for some $\omega \in \T$, as a solution of the equation \eqref{eq4.30}. Hence the construction of Theorem \ref{thm4.10} yields $h$ for the appropriate
 choices of $t_+ >0$, $t \in \R \setminus \{0 \}$ and $\omega$. 
\end{proof}

\section{$s$-Convexity and $s$-extremity} \label{convexity}

The distinguished boundaries of the bidisc $\D^2$ and the ball $\mathbb{B}_2$ contain no line segments. Every inner function in $\hol(\D,\D^2)$ is therefore an extreme point of $\hol(\D,\D^2)$, and likewise for $\hol(\D,\mathbb{B}_2)$.  This property contrasts sharply with the situation in the symmetrized bidisc.

 $\Gamma$ is not a convex set, but it is convex in $s$ for fixed $p\in\D^-$.  That is, the set
\beq\label{gamp}
\Gamma \cap  (\C\times \{p\}) = \{(s,p)\in\C^2: |s-\bar s p|\leq 1-|p|^2 \mbox { and } |s|\leq 2\}
\eeq
is convex for every $p\in\D^-$, as is easily seen from the expression on the right hand side of equation \eqref{gamp}.
In consequence, some associated sets have a similar property.
\begin{proposition}\label{sconvsets}
The following sets are convex.
\begin{enumerate}
\item  $\Gamma \cap  (\C\times \{p\})$ for any $p\in\D^-$;
\item  $\M \cap  (\C\times \{p\})$ for any $p\in\T$;
\item the set of $\Gamma$-inner functions $(s,p)$ for a fixed inner function $p$.
\end{enumerate}
\end{proposition}
To prove (2) observe that $\M \cap  (\C\times \{p\})=\{(s,p): s= \bar s p \mbox { and } |s|\leq 2\}$.  Statement (3) follows easily from the first two. 

We shall summarize these properties by saying that $\Ga, \M$ and the set of $\Ga$-inner functions are {\em $s$-convex}.

In the light of the phenomenon of $s$-convexity it is natural to ask about the extreme points of the set (3) in Proposition \ref{sconvsets}.
\begin{defin}\label{def5.10}
A rational $\gaminn$ function $h$ is {\em $s$-extreme} 
if whenever $h$ has a representation of the form  $h = t h_1 + (1-t) h_2$ with $t \in (0,1)$ and $h_1$ and $h_2$ rational $\gaminn$ functions, $h_1=h_2$.
\end{defin}
Thus $h$ is {$s$-extreme} if and only if it is an extreme point of the set of rational $\Gamma$-inner functions in the usual sense; however, one customarily only speaks of extreme points of {\em convex} sets, and the rational $\Gamma$-inner functions do not constitute a convex set.  It is thus safer to use the term $s$-extreme.   

We show in this section that whether or not a rational $\Ga$-inner function is $s$-extreme depends entirely on how many royal nodes it has in $\T$ (Theorem \ref{2k>n}).

 It follows from the lemma below that a $\Ga$-inner function is $s$-extreme if and only if it is an extreme point of the set in (3) of Proposition \ref{sconvsets} for some inner function $p$.

\begin{lemma}\label{extr_p} 
Let $h=(s,p)$, $h_1=(s_1,p_1)$ and $h_2=(s_2,p_2)$  be $\gaminn$ functions.
If $h =t h_1 + (1-t) h_2$ for some $t$ such that $0 < t < 1$ then $p = p_1= p_2$.
\end{lemma}
\begin{proof}  Since $h =t h_1 + (1-t) h_2$ we have $p =t p_1 + (1-t) p_2$. 
Hence, at any point $\lambda \in \t$, $p(\lambda) =t p_1(\lambda) + (1-t) p_2(\lambda)$. Since the functions $p, p_1, p_2 $ are inner, $p(\la)\in\T$ and both $p_1(\la)$ and $p_2(\la)$ are in $\D^-$.  Since every point of $\t$ is an extreme point of $\bar{\D}$ we have $p(\la) = p_1(\la)= p_2(\la)$ for almost all $\la\in\T$.
\end{proof}

\begin{lemma}\label{lem5.10}
Let $h=(s,p)$ be a rational $\gaminn$ function. For $\tau \in \t$,  $|s(\tau)| =2$ if and only if $\tau$ is a royal node of $h$.  Moreover, $\tau=\e^{it_0}$ is a royal node of $h$ of multiplicity $\nu$ if and only if $| s(\e^{it})|=2 $ to order $2\nu$ at $t=t_0$.
\end{lemma}
Here a {\em (real or complex-valued) function $f$ on a real interval $I$ is said  to take a value $y$ to order $m\geq 1$ at a point $t_0\in I$}
if $f \in C^m(I), \, f(t_0)=y, \,  f^{(j)}(t_0)=0$ for $j=1,\dots, m-1$ and $f^{(m)}(t_0)\neq 0$. Note that if  $y\neq 0$ then  $f^2(t_0)= y^2$ to order $m$ implies that $f(t_0)=y$ to order $m$. We say that {\em $f$ vanishes to order $m\geq 1$ at a point $t_0\in I$}
if $f$ take the value $0$ to order $m$ at $t_0$.\\
\begin{proof}
By Definition \ref{def3.30}, to say that $\tau\in\T$ is a royal node of $h$ of multiplicity $\nu$ means that 
\[
(s^2-4p)(\la) = (\la-\tau)^{2\nu} F(\la)
\]
for some rational function $F$ that is analytic on $\T$ and does not vanish at $\tau$.  

Since $h$ is $\Ga$-inner, $s=\bar sp$ on $\T$, and hence
\[
4-|s|^2 = \frac{1}{p}(4p -s\bar s p) = -\frac{1}{p}(s^2-4p)
\]
on $\T$.  It is immediate that, for any $\tau\in\T, \, |s(\tau)|=2$ if and only if $s(\tau)^2=4p(\tau)$, that is, if and only if $\tau$ is a royal node of $h$.

Now suppose that $\tau=\e^{it_0}$ is a royal node of $h$ of multiplicity $\nu\geq 1$.  On combining the last two displayed formulae one finds that, for all $t\in\R$,
\[
4-|s(\e^{it})|^2 = (e^{it} -\tau)^{2\nu} G(e^{it})
\]
where $G=-F/p$ is a rational function that is analytic on $\T$ and does not vanish at $\e^{it}=\tau$.  Since $h$ is rational and $|s(\e^{it_0})|=2$, the function $f(t)=4-|s(\e^{it})|^2$ is $C^\infty$ on a neighbourhood of $t_0$.  It is routine to show that $f^{(j)}(t_0)$ is zero for $j=0,\dots,2\nu-1$ and nonzero when $j=2\nu$.  Thus $f(t)=0$ to order $2\nu$ at $t_0$.  Hence $|s(\e^{it})|^2=4$ to order $2\nu$ at $t_0$, and so $|s(\e^{it})|=2$ to order $2\nu$ at $t_0$.
\end{proof}

\begin{lemma}\label{Rn0} Any $h = (s,p) \in \royal^{n,0}$ is not $s$-extreme.
\end{lemma}

\begin{proof} As the royal nodes of $h$ all lie in $\d$, by Lemma \ref{lem5.10},  
$|s| < 2$ on $\t$. Hence, there exists $\epsilon > 0$ such that $|s+\epsilon s|<2$ on $\t$. It follows from Proposition \ref{prop2.10} that if we define $h_1$ and $h_2$ by $h_1 = (s+\epsilon s,p)$ and $h_2 = (s-\epsilon s,p)$, then $h_1$ and $h_2$ are rational $\gaminn$ functions. Furthermore, an application of Lemma \ref{lem5.10} to $h_1$ and $h_2$ shows that both $h_1,h_2 \in \royal^{n,0}$. Finally, since by construction, $h = \frac{1}{2} h_1 + \frac{1}{2} h_2$, the proof of Lemma \ref{Rn0} is complete.
\end{proof}

Recall from the introduction that a superficial $\Ga$-inner function $h$ is one such that $h(\D) \subset \partial \Ga$ and that they are all of the form $(\omega +\bar \omega p,p)$ for some inner function $p$ and some $\omega\in\T$ \cite[Proposition 8.3]{ALY12}.

\begin{proposition}\label{superficial}

{\rm (i)} Let $h=(s,p)$ be superficial 
and  $h =  t h_1 + (1-t) h_2$ for some $0 < t < 1$, where 
$h_1=(s_1,p_1)$ and $h_2=(s_2,p_2)$ are rational $\G$-inner functions.
Then $h_1$ and $h_2$ are superficial and $p = p_1= p_2$.
 
{\rm (ii)} Superficial $\gaminn$ functions are $s$-extreme.

{\rm (iii)} Suppose $h_1= (\omega_1 +\bar{\omega_1} p, p)$ and $h_2= (\omega_2 +\bar{\omega_2} p, p)$ are superficial $\gaminn$ functions of degree $n$ such that $\omega_1 \neq \omega_2$
and $h =  t h_1 + (1-t) h_2$ for some $0 < t < 1$. Then $h  \in \royal^{n,0}$ and is not $s$-extreme.
\end{proposition} 
\begin{proof} (i) Suppose $h_1$ is not superficial. Then there exists $\la_0 \in \D$ such that $h_1(\la_0) \in \G$. Let us show that in this case $h(\la_0) \in \G$. By Lemma \ref{extr_p},  $p = p_1= p_2$.
By Proposition \ref{prop2.10}, it is enough to show that
\[
|s(\la_0)- \bar{s(\la_0)} p(\la_0)| <  1 - |p(\la_0)|^2. 
\]
Note that 
\[
|s(\la_0)- \bar{s(\la_0)} p(\la_0)| = t| (s_1(\la_0)- \bar{s_1(\la_0)} p(\la_0))| + (1-t)|(s_2(\la_0)- \bar{s_2(\la_0)} p(\la_0)) |
\]
\[
 <  1 - |p(\la_0)|^2. 
\]

(ii) By \cite[Proposition 8.3]{ALY12}, a superficial rational $\gaminn$ function $h=(\omega +\bar \omega p,p)$ for some inner function $p$ and some $\omega\in\T$. If $h$ is not  $s$-extreme, by Part(i), we have $h =  t h_1 + (1-t) h_2$ for some $t$ such that $0 < t < 1$, where 
$h_1=(s_1,p)$ and $h_2=(s_2,p)$ are superficial rational $\G$-inner functions. Let  $h_i=(\omega_i +\bar \omega_i p,p)$ for  some $\omega_i\in\T$, $i=1,2$. Thus
\[
h=(\omega +\bar \omega p,p)= (t \omega_1 + t \bar \omega_1 p +
(1-t)\omega_2 +(1-t)\bar \omega_2 p, p).
\]
Therefore, for $\omega \in \T$ and $\omega_i\in\T$, $i=1,2$,
\[
\omega= t \omega_1 + 
(1-t)\omega_2.
\]
Since every point of $\T$ is an extreme point of $\bar \D$ we have 
\[
\omega= \omega_1 =\omega_2
\]
and $h = h_1= h_2$. Hence $h$ is $s$-extreme.

(iii) Suppose $h_1= (\omega_1 +\bar{\omega_1} p, p)$ and $h_2= (\omega_2 +\bar{\omega_2} p, p)$ are superficial $\gaminn$ functions of degree $n$ such that $\omega_1 \neq \omega_2$
and $h =  t h_1 + (1-t) h_2$ for some $0 < t < 1$. Thus
\[
h = (\omega +\bar\omega p, p)
\]
where $\omega= t \omega_1 + (1-t)\omega_2$.
Since $\omega_1 \neq \omega_2$ we have $|\omega| <1$ and $h$ has no royal nodes on $\T$.
Therefore $h  \in \royal^{n,0}$ and, by Lemma \ref{Rn0}, $h$ is not $s$-extreme.
\end{proof}

\begin{proposition}\label{complex_geodesic} Let $h$ be the $\gaminn$ function
\[
h(\lambda) = (\beta +\bar{\beta} \lambda, \lambda)
\]
where $|\beta| \le 1$.\\
{\rm (i)} If  $|\beta| < 1$ then $h \in \royal^{1,0}$ and $h$ is not $s$-extreme.\\
{\rm (ii)} If  $|\beta| = 1$ then $h \in \royal^{1,1}$ and $h$ is  $s$-extreme.
\end{proposition} 

Note that if $|\beta|<1$ then $h$ is a complex geodesic of $\G$ (it has the analytic left inverse $(s,p)\mapsto p$).
If $|\beta|=1$ then $h$ is superficial, and so is not a complex geodesic of $\G$.\\
\begin{proof}
(i) If $|\beta| < 1$ then, for all $\lambda \in \t$, $|s(\lambda)| \le 2|\beta| <2$. By Lemma \ref{lem5.10},
 $h$ has no royal node on $\t$ and so $h \in \royal^{1,0}$. By Lemma \ref{Rn0},  $h$ is not $s$-extreme.

(ii)  Let $|\beta|=1$.  Then $|s(\la)|=2$ if and only if $\la=\beta^2$.  Hence the royal node of $h$ is at $\beta^2 \in\T$.  Hence $h \in \royal^{1,1}$.  By Proposition \ref{superficial}(ii), $h$ is  $s$-extreme.
\end{proof}

For $p$ an inner function of degree $n$ and $k=0,1,\ldots,n$, let
\be\label{eq5.20}
\royal_p^{n,k}=\set{(s,p_1) \in \royal^{n,k}}{p_1=p}
\ee
and let $\royal_p^n$ be the set of $\Ga$-inner functions with second component $p$, so that
\be\label{eq5.30}
\royal_p^n=\bigcup_{k=0}^n \royal_p^{n,k}.
\ee

\begin{proposition}\label{prop5.10}
If $p$ is an inner function of degree $n$ then $\royal_p^n$ is convex.
Let $C$ be a collection of rational $\gaminn$ functions. $C$ is convex if and only if there exists an inner function $p$ such that $C$ is a convex subset of $\royal_p^n$.
\end{proposition}
\begin{proof} It follows from Proposition \ref{sconvsets} and Lemma \ref{extr_p}.
\end{proof}

\begin{proposition}\label{prop5.20}
If $h$ is a rational $\gaminn$ function of degree $n$ then $h$ is a convex combination of at most $n+2$ $s$-extreme rational $\gaminn$ functions of degree at most $n$.
\end{proposition}
\begin{proof} For the given rational $\gaminn$ function $h=(s,p)$ 
of degree $n$, $p$ is an inner function of degree $n$ and $h \in \royal_p^n$. By Remark \ref{dimE}, the convex  set $\royal_p^n$ is a subset of an  ($n+1$)-dimensional real subspace of the rational functions. Thus, by a theorem of Carath$\acute{\rm e}$odory \cite{Ca,St}, $h$ is a convex combination of at most $n+2$ $s$-extreme rational $\gaminn$ functions of degree at most $n$.
\end{proof}
\begin{lemma}\label{s-convex-nodes} Let $h=(s,p) \in \royal_p^{n,k}$
and let $\tau_1, \tau_2,  \dots, \tau_k \in \T$ be royal nodes of $h$. Suppose $h =  t h_1 + (1-t) h_2$ for some $t$ such that $0 < t < 1$, where 
$h_1=(s_1,p_1)$ and $h_2=(s_2,p_2)$ are rational $\G$-inner functions.
Then  $p = p_1= p_2$ and
\[ 
s_i(\tau_j) = s(\tau_j) \; \; \text{for } \;  j =1,\dots k\; \text{and} \; i=1,2.
\]
\end{lemma}
\begin{proof} By Lemma \ref{extr_p},  $p = p_1= p_2$. By Lemma \ref{lem5.10}, $|s(\tau_j)| =2$ at each royal node $\tau_j \in \T$.
Note that $s(\tau_j) =  t s_1(\tau_j) + (1-t) s_2(\tau_j)$ for some  $t$ such that  $0 < t < 1$  and $|s_i(\tau_j)| \le 2$  for $ j =1,\dots,k$ and $i=1,2$.
Since every point on the circle $2 \T$ is an extreme point of $2 \bar{\d}$ we have $|s_i(\tau_j)| = 2$  and
$ s_i(\tau_j) = s(\tau_j)$ for  $ j=1,\dots, k$ and $i=1,2$.
\end{proof}

The following observation follows easily from a consideration of Taylor expansions.
\begin{lemma}\label{vanishes}
Let $t_0< t_1$ in $\R$, let $\nu\geq 1$ be an integer and let $f,g$ be nonnegative real-valued $C^{2\nu}$ functions on $[t_0,t_1)$.  If $f$ vanishes to order $2\nu$ at $t_0$ and $0\leq g\leq f$ on $[t_0,t_1)$ then $g$ vanishes to order at least $2\nu$ at $t_0$.
\end{lemma}
\begin{lemma}\label{tricky}
If $h,h_1$ and $h_2$ are rational $\Ga$-inner functions, $h$ is a convex combination of $h_1$ and $h_2$ and $\tau\in\T$ is a royal node of $h$ of multiplicity $\nu>0$ then $\tau$ is also a royal node of multiplicity at least $\nu$ for $h_1$ and $h_2$. 
\end{lemma}
\begin{proof}
Let $h=(s,p), \, h_1=(s_1,p_1), \, h_2=(s_2,p_2)$.  Suppose that $h=\half h_1+\half h_2$.  By Lemma \ref{extr_p}, $p_1=p=p_2$, \, $s=\half s_1+\half s_2$ and we can assume that $s, s_1$ and $s_2$ are distinct rational functions. 

 The argument can be pictured as follows.  Imagine the closed curve $h(\e^{it}), 0\leq t\leq 2\pi$, lying in the M\"obius band $\M$.  It touches the boundary $\{(2\omega,\omega^2):\omega\in\T\}$ of $\M$ at the points where $\tau = \e^{it_0}$ is a royal node of $h$ lying in $\T$, and by Lemma \ref{lem5.10}, it touches to order $2\nu$, where $\nu$ is the multiplicity of the royal node in question.  For each $t$, the three points $h(\e^{it}), h_1(\e^{it})$ and $h_2(\e^{it})$ lie on the line segment
\[
L_\theta= \{(2x \e^{i\theta/2}, \e^{i\theta}):-1 \leq x\leq 1\} \quad\subset\quad \M
\]
where $p(\e^{it})=\e^{i\theta}$, and furthermore the first of these three points is the midpoint of the other two. The three curves intersect only finitely many times, for otherwise they coincide.  Hence, for $t$ in a small enough one-sided neighbourhood of $t_0$, one of the curves, say $h_1(\e^{it})$, is sandwiched between $h(\e^{it})$ and the appropriate endpoint $(\pm 2 \e^{i\theta/2}, \e^{i\theta})$ of $L_\theta$.  We shall show with the aid of Lemma \ref{vanishes} that $h_1(\e^{it})$ and $h_2(\e^{it})=(2h-  h_1)(\e^{it})$ also touch the boundary to order at least $2\nu$.   Hence $h_1, \, h_2$ have royal nodes at $\tau$, and (again by Lemma \ref{lem5.10}) with  multiplicities at least $\nu$.

Let us formalise this geometric argument.
Suppose that $\tau= \e^{it_0}$ and $p(\tau)=\e^{i\theta_0}$.  Let $I$ be an open interval in $\R$ containing $t_0$ and such that $\exp(iI)$ has length less than $2\pi$; thus there is an analytic branch of $\log$ on the arc $\exp(iI)$.   Define a chart $(U, \al)$ in $\M$  by taking $U$ to be the set $\M \cap (\C\times \exp(iI))$ and 
\[
\al=(X,\Theta):U\to \R^2
\]
 to be defined by 
\[
\al(s,\e^{i\theta}) =(\half s \e^{-i\theta/2}, \theta)
\]
where of course the map $\e^{i\theta}\mapsto \theta$ is $-i\log$.  Observe four properties of $\al$:
\begin{enumerate}
\item the image $\al(U)$ is the rectangle $[-1,1]\times I$;
\item $\al$ is real affine linear on every line segment $L_\theta$;
\item $X$ is real-valued on $U$ and so, for $(s,\e^{i\theta})\in U$,
\[
X(s,\e^{i\theta}) =\pm |X(s,\e^{i\theta})|=\pm|\half s \e^{-i\theta/2}|= \pm \half |s|.
\]
\item 
\[
\al(s,p) =\left(\frac{s}{2 \sqrt{p}}, -i \log p \right)
\]
is the restriction to $U$ of an analytic map on an open set in $\C^2$. Thus $X$ is real-analytic in $U$.
\end{enumerate}

The point $h(\tau)$ lies the boundary of $b\Gamma$ and is therefore of the form $(2\omega,\omega^2)$ for some $\omega\in\T$.  Here $\omega^2=p(\tau)=\e^{i\theta_0}$ and so $\omega=\pm \e^{i\theta_0/2}$.
If $I$ is replaced by $I+2\pi$ then $U$ is unchanged
and the sign of the first component of $\al$ is reversed; we may therefore assume that $\omega=\e^{i\theta_0/2}$ and $\al\circ h(\tau)=(1,\theta_0)$.  On replacement of $I$ by a smaller neighbourhood of $t_0$ if necessary, we can also assume that $X\circ h(\e^{it}) >0$  and so $X\circ h(\e^{it}) = \half |s(\e^{it})|$ for $t\in I$.  Similarly we can assume that 
$X\circ h_j(\e^{it}) = \half |s_j(\e^{it})|$ for $t\in I$ and $j=1,2$.
Let 
\[
f(t)= 1-X\circ h(\e^{it})= 1-\half |s(\e^{it})|
\]
for $t\in I$.  Likewise let $f_j(t)= 1-X\circ h_j(\e^{it})= 1-\half |s_j(\e^{it})|$ for $j=1,2$.  Then $f,f_1$ and $f_2$ are all nonnegative on $I$ and, by the affine linearity property of $\al$,
\[
f= \half f_1+\half f_2.
\]
By hypothesis $\tau$ is a royal node of $h$ of multiplicity $\nu$, and so, by Lemma \ref{lem5.10},
$|s(\e^{it})|=2$ to order $2\nu$ at $t_0$, which is to say that $f$ vanishes to order $2\nu$ at $t_0$.  Now
$f,f_1,f_2$ are distinct at all but finitely many points.  Hence there is a  $t_1>t_0$ contained in $I$ such that
$0\leq f_j(t)\leq f(t)$ for $t_0<t<t_1$ and $j=1$ or $2$ -- say $j=1$.  By Lemma \ref{vanishes}, it follows that $f_1$ and also $f_2=2f-f_1$ vanish to order  at least  $2\nu$ at $t_0$.   Consequently $|s_j(\e^{it})|=2$ to order at least $2\nu$ at $t_0$.  Again by Lemma \ref{lem5.10}, $h_j$ has a royal node of multiplicity at least $\nu$ at $\tau$.
\end{proof}
We are ready to prove Theorem \ref{2k>n}.  Recall the statement:
\begin{theorem}\label{2k>nbis}
A $\Ga$-inner function of degree $n$ having $k$ royal nodes in $\T$, counted with multiplicity, is $s$-extreme if and only if $2k>n$.
\end{theorem}
In other words, $h\in\mathcal{R}^{n,k}$  is $s$-extreme if and only if $2k > n$.\\
\begin{proof}
($\Rightarrow$)  Let $h\in \mathcal{R}^{n,k}$.  Suppose that $2k \leq n$; we must show that $h$ is not $s$-extreme.
Write $h$ in polynomial form: $h=(E,D^{\sim n})/D$ where $E$ is an $n$-symmetric polynomial and $D$ is a polynomial of degree at most $n$ having no zeros in $\D^-$.  Let the royal nodes of $h$ in $\D^-$ be $\tau_1, \dots,\tau_k \in\T$ and
$\al_{k+1},\dots,\al_n \in \D$ (with repetitions according to multiplicity).  The royal polynomial of $h$ is then
\[
R= \prod_{j=1}^k Q_{\tau_j}\prod_{j=k+1}^n Q_{\al_j},
\]
and consequently
\[
\la^{-n} R(\la) = \prod_{j=1}^k |\la-\tau_j|^2 \prod_{j=k+1}^n |\la- \al_j|^2
\]
for all $\la \in\T$.
By Theorem \ref{thm4.10},
\beq\label{4DEt0}
4|D|^2-|E|^2 = r \prod_{j=1}^k |\la-\tau_j|^2 \prod_{j=k+1}^n |\la- \al_j|^2
\eeq
for some $r>0$ and all $\la\in\T$.

Assume first that $n$ is even, say $n=2m$.   Thus $k\leq m$.  Let
\[
g(\la)= \bar\tau_1 \dots \bar\tau_k \la^{m-k}\prod_{j=1}^k (\la-\tau_j)^2.
\]
This polynomial has degree  $m+k\leq n$ and is $n$-symmetric.
Let $E_t=E+tg$ for $t\in\R$.  Then $E_t$ is $n$-symmetric of degree at most $n$, and
\beq\label{4DEt}
4|D|^2 - |E_t|^2 = 4|D|^2 - |E|^2 - t^2|g|^2- 2 \re (t\bar E g)
\eeq
on $\T$.  Let $\|E\|_\infty $ denote the supremum of $|E|$ on $\T$; then
\begin{align*}
\re(t\bar E g(\la)) &\leq |t \bar E g(\la)| = |tE(\la)| \prod_{j=1}^k |\la-\tau_j|^2 \\
	&\leq |t| \|E\|_\infty   \prod_{j=1}^k |\la-\tau_j|^2
\end{align*}
for $\la\in\T$.  Combine this inequality with equations \eqref{4DEt0} and \eqref{4DEt} to deduce that
\begin{align*}
4|D|^2 - |E_t|^2 &\geq  r\prod_{j=1}^k |\la-\tau_j|^2 \prod_{j=k+1}^n |Q_{\la_j}(\la)|  - t^2|g|^2 -2|t|\|E\|_\infty \prod_{j=1}^k|\la-\tau_j|^2 \\
&=\prod_{j=1}^k |\la-\tau_j|^2 \left\{ r\prod_{j=k+1}^n |Q_{\al_j}(\la)| -(t^2+2|t|\|E\|_\infty) \right\} \\
&\geq \prod_{j=1}^k |\la-\tau_j|^2 \left\{ r M-(t^2+2|t|\|E\|_\infty) \right\}
\end{align*}
on $\T$, where $M = \inf_{\T} \prod |Q_{\al_j}| > 0$.  It follows that for $|t|$ sufficiently small, $4|D|^2 - |E_t|^2 \geq 0$ on $\T$.  Hence, by Theorem \ref{thm4.10}, the functions
\[
h_{\pm t} \df \left( \frac{E_{\pm t}}{D}, \frac{D^{\sim n}}{D}\right)
\]
are rational $\Ga$-inner functions, and clearly $h=\half h_t+\half h_{-t}$.  Thus $h$ is not $s$-extreme.

The case of odd $n$, say $n=2m+1$, requires a slight modification.  Since $2k\leq n=2m+1$, we have $k\leq m$.
Choose $\omega\in\T$ such that 
\[
\omega^2=-\bar\tau_1 \prod_{j=1}^k \bar\tau_j^2
\]
and let
\[
g(\la) = \omega \la^{m-k} (\la-\tau_1)\prod_{j=1}^k (\la-\tau_j)^2.
\]
Then $g$ is an $n$-symmetric polynomial of degree $m+k+1 \leq n$.  As in the even case we define $E_t$ to be $E+tg$ for real $t$, and a similar calculation to the foregoing shows that $4|D|^2-|E_t|^2 \geq 0$ on $\T$ for small enough $|t|$.  The argument concludes as before to show that $h$ is not $s$-extreme. 

($\Leftarrow$) 
 Let $2k>n$ and suppose  that $h=(s,p)=(E/D,D^{\sim n}/D)$ is not $s$-extreme, so that there exist $n$-symmetric polynomials $E_\pm$ of degree at most $n$, different from $E$, such that $h=\half h_++\half h_-$ where
\[
h_\pm=(s_\pm,p)= \left(\frac{E_\pm }{D}, \frac{D^{\sim n}}{D}\right)
\]
are $\Ga$-inner functions.
Let the royal nodes of $h$ in $\T$ be $\tau_1,\dots, \tau_\ell$ with multiplicities $\nu_1,\dots,\nu_\ell$ respectively.   Thus $\nu_1+\dots+\nu_\ell=k$.   Let $g=E_+-E$; then $g$ is a nonzero $n$-symmetric polynomial  and $E_-= E- g$.  

Let the royal polynomials of $h$ and $h_\pm$ be $R$ and $R_\pm$ respectively.  Then
\[
R_\pm=4DD^{\sim n} -(E\pm g)^2 = R -g^2\mp 2 E g.
\]
Hence 
\[
R_- -R_+=4Eg.
\]

By Lemma \ref{tricky}, $h_\pm$ have royal nodes of multiplicity at least $\nu_j$ at $\tau_j$, and so $R_\pm$ vanish to order $2\nu_j$ at $\tau_j$. 
Hence $g$ vanishes to order at least $2\nu_j$ at $\tau_j$.  The degree of $g$ is therefore at least $2\nu_1+\dots +2\nu_\ell = 2k >n$.  This is a contradiction since $\deg (g) \leq n$, and so $h$ is $s$-extreme.

\end{proof}

JIM  ~ AGLER, Department of Mathematics, University of California at San Diego, CA \textup{92103}, USA\\

ZINAIDA A. LYKOVA,
School of Mathematics, Statistics and Physics, Newcastle University, Newcastle upon Tyne
 NE\textup{1} \textup{7}RU, U.K.~~\\

N. J. YOUNG, School of Mathematics, Statistics and Physics, Newcastle University, Newcastle upon Tyne NE1 7RU, U.K.
{\em and} School of Mathematics, Leeds University,  Leeds LS2 9JT, U.K.

\begin{thebibliography}{50} \label{bibliog}
\bibitem{AY99} J. Agler and N. J. Young, A commutant lifting theorem for a domain in
${\mathbb C}^2$ and spectral interpolation, {\em J. Functional
Analysis} {\bf 161} (1999) 452--477.

\bibitem{ALY12}  J. Agler, Z. A. Lykova and N. J. Young, Extremal holomorphic maps and the symmetrized bidisc,  {\em  Proc. London Math. Soc.}  {\bf 106}(4) (2013)  781--818.

\bibitem{ALY13_2} J. Agler, Z. A. Lykova and N. J. Young,  $3$-extremal holomorphic maps and the symmetrised bidisc, {\em J. Geom. Anal.} {\bf 25} (2015) 2060--2102.

\bibitem{AY04}  J. Agler and N. J. Young, The hyperbolic geometry of the symmetrized bidisc, {\em J. Geom. Anal.} {\bf 14} (2004) 375--403.

\bibitem{AY04T} J. Agler and N. J. Young,  The two-by-two
spectral Nevanlinna-Pick problem, {\em Trans. Amer. Math. Soc.}
 {\bf 356} (2004)  573--585.


\bibitem{AY06}  J. Agler and N. J. Young, The complex geodesics of the symmetrized bidisc, {\em Inter. J. of Mathematics} {\bf 17} no.4 (2006) 375--391.

\bibitem{AY08}
J. Agler and N. J. Young,
The magic functions and automorphisms of a
domain, {\em Complex Analysis and Operator Theory} {\bf 2} (2008) 383--404.

\bibitem{Ca} C. Carath\'eodory, \"Uber den Variabilit$\ddot{\rm a}$tsbereich der Fourierschen Konstanten von positiven harmonischen Funktionen, {\em Rendiconti del Circolo Matematico di Palermo} {\bf 32} (1911) 193-217.


\bibitem{costara} C. Costara,  {\em Le probl\`eme de Nevanlinna-Pick spectral}, PhD thesis, Laval University (2004), http://ariane.ulaval.ca/cgi-bin/recherche.cgi?qu=a1449522  .

\bibitem{Di89} S. Dineen, {\em The Schwarz Lemma}, Oxford University Press, 248 pages,
(1989).

\bibitem{Do} J. C. Doyle, Analysis of feedback systems
with structured  uncertainties. {\em  IEE Proceedings} {\bf 129} (1982), no. 6, 242--250.

\bibitem{dullerud} G. Dullerud and F. Paganini, {\em A course in robust control theory: a convex approach}, Texts in Applied Mathematics {\bf 36}, Springer, 2000.

\bibitem{KZ2014}  L. Kosinski and W. Zwonek,  Extremal holomorphic maps in special classes of domains, {\em Annali della Scuola Normale Superiore di Pisa, Classe di Scienze},
{\bf 16} (2016), issue 1,  159--182.

\bibitem{mathworks}
MathWorks Inc., {\em Robust Control Toolbox}, Natick, Massachusetts, U.S.A.

\bibitem{ogle}  D. Ogle,  Operator and function theory of the symmetrised polydisc, PhD thesis, Newcastle University, 1999.   https://theses.ncl.ac.uk/dspace/bitstream/10443/1264/1/Ogle 99.pdf

\bibitem{RN}
F. Riesz and B. Sz.-Nagy.
\newblock {\em Functional Analysis}.
\newblock  Dover, New York, 1990.

\bibitem{roro85}
M. Rosenblum and J. Rovnyak, {\em Hardy Classes and Operator Theory}, Oxford Mathematical Monographs, OUP 1985.

\bibitem{St} E. Steinitz, Bedingt konvergente Reihen und konvexe Systeme, {\em J. Reine Angew. Math.}  {\bf 143} (1913) 128--175.\\

\end{thebibliography}
\end{document}